\documentclass{article}
\usepackage{ijcai20}

\usepackage{amsmath,amsfonts,bm}

\def\eqref#1{equation~\ref{#1}}

\def\1{\bm{1}}

\DeclareMathAlphabet{\mathsfit}{\encodingdefault}{\sfdefault}{m}{sl}
\SetMathAlphabet{\mathsfit}{bold}{\encodingdefault}{\sfdefault}{bx}{n}

\usepackage{times}
\usepackage{soul}
\usepackage{url}
\usepackage[hidelinks]{hyperref}
\usepackage[utf8]{inputenc}
\usepackage[small]{caption}
\usepackage{graphicx}
\usepackage{amsmath}
\usepackage{amsthm}
\usepackage{booktabs}
\usepackage{algorithm}
\usepackage{algorithmic}
\usepackage{multirow}

\allowdisplaybreaks[3]

\usepackage{dirtytalk}

\usepackage{url}            
\usepackage{booktabs}       
\usepackage{amsfonts}       
\usepackage{nicefrac}       
\usepackage{microtype}      
\usepackage{comment} 
\usepackage{lipsum} 
\usepackage{graphics} 
\usepackage{epsfig} 
\usepackage{times} 
\usepackage{amsmath} 
\usepackage{amssymb}  
\usepackage{amsthm} 
\usepackage{subcaption}
\usepackage{color}
\usepackage{xcolor}
\usepackage{caption}
\newcommand{\norm}[1]{\left\lVert#1\right\rVert}
\newtheorem{theorem}{Theorem}

\newtheorem{lemma}{Lemma}
\newtheorem{remark}{Remark}

\newtheorem{assumption}{Assumption}
\usepackage{enumitem}

\usepackage{comment,cleveref}

\title{Analysis of Q-learning with Adaptation and Momentum Restart \\ for Gradient Descent}

\author{
Bowen Weng\thanks{Equal contribution}$^1$
\and
Huaqing Xiong$^{*1}$\and
Yingbin Liang$^1$\And
Wei Zhang\thanks{Corresponding author}$^2$\thanks{This paper is an extended work from a preliminary version presented at the 2020 International Joint Conferences on Artificial Intelligence}
\affiliations
$^1$Electrical and Computer Engineering, The Ohio State University, Columbus, OH, USA.\\
$^2$Mechanical and Energy Engineering, Southern University of Science and Technology, China.
\emails
\{weng.172, xiong.309, liang.889\}@osu.edu,
zhangw3@sustech.edu.cn
}

\begin{document}

\maketitle

\begin{abstract}
Existing convergence analyses of Q-learning mostly focus on the vanilla stochastic gradient descent (SGD) type of updates. Despite the Adaptive Moment Estimation (Adam) has been commonly used for practical Q-learning algorithms, there has not been any convergence guarantee provided for Q-learning with such type of updates. In this paper, we first characterize the convergence rate for Q-AMSGrad, which is the Q-learning algorithm with AMSGrad update (a commonly adopted alternative of Adam for theoretical analysis). To further improve the performance, we propose to incorporate the momentum restart scheme to Q-AMSGrad, resulting in the so-called Q-AMSGradR algorithm. The convergence rate of Q-AMSGradR is also established. Our experiments on a linear quadratic regulator problem show that the two proposed Q-learning algorithms outperform the vanilla Q-learning with SGD updates. The two algorithms also exhibit significantly better performance than the DQN learning method over a batch of Atari 2600 games.
\end{abstract}

\section{Introduction} \label{sec:intro}
Q-learning~\cite{watkins1992q}, as one of the most important model-free reinforcement learning (RL) algorithms, has received considerable attention in recent years~\cite{bertsekas1996neuro,even2003learning,tyler2018nd}. The vanilla Q-learning algorithm runs a step of empirical Bellman operator update of the Q-function and a step of stochastic gradient descent (SGD) in an \textit{\textbf{alternating}} fashion. The convergence guarantee of Q-learning has been studied for the tabular case in \cite{bertsekas1995dynamic,even2003learning}, for the case with linear function approximation in~\cite{bertsekas1996neuro,zou2019finite,Chen2019finiteQ}, and also for neural network parameterization in~\cite{xu2019deepQ}.

However, all the existing theoretical analyses focus on Q-learning algorithms that take simple SGD iterations. Such theory is not applicable to practical Q-learning algorithms that implement the Adaptive Moment Estimation (Adam) type of updates. In this paper, we study the Q-learning algorithm with Adam-type updates in terms of its theoretical convergence and the performance in benchmark experiments. It is known in optimization that Adam does not always converge, and instead, and a slightly modified variant AMSGrad proposed in~\cite{reddi2019convergence} has been widely accepted as an alternative to justify the theoretical performance of Adam-type methods. This motivates the first question that we aim to address.
\begin{list}{$\bullet$}{\topsep=0.ex \leftmargin=0.27in \rightmargin=0.in \itemsep =0.02in}
\item {\em Q1: Can we provide the convergence guarantee for Q-learning under AMSGrad updates (i.e., Q-AMSGrad)?}
\end{list}

In conventional optimization problems, {\em restart} has been incorporated into the gradient descent algorithm with momentum as a simple yet effective scheme to facilitate the acceleration performance \cite{o2015adaptive}. 
Hence, it is natural to incorporate such a restart technique into Q-learning and ask the following question about its performance.
\begin{list}{$\bullet$}{\topsep=0.ex \leftmargin=0.27in \rightmargin=0.in \itemsep =0.02in}
\item {\em Q2: Does Q-AMSGrad with momentum restart (i.e., Q-AMSGradR) still converge?}
\end{list}

As aforementioned, both Q-AMSGrad and Q-AMSGradR update alternatingly between one step of Bellman operator update of the Q-function and one step of adaptive momentum update. This is in contrast to the well-known deep Q-Network (DQN) learning~\cite{mnih2015human}, which runs in a nested-loop manner with the outer loop consisting of an one-step update of the Q-function and the inner loop consisting of {\em many} iterations of supervised learning to fit a {\em target Q-function}. It is conventionally known that taking {\em just one} gradient step toward the target Q-function results in the instability of original Q-learning update. Thus the simple alternating update manner mainly aroused interest in theory instead of practice. We are interested in whether the adaptive momentum updates can improve the stability of Q-learning without using the supervised learning process for target network fitting.
Therefore, the third question we want to address is to compare these two types of Q-learning algorithms from the experimental perspective.
\begin{list}{$\bullet$}{\topsep=0.ex \leftmargin=0.27in \rightmargin=0.in \itemsep =0.02in}
\item {\em Q3: Do Q-AMSGrad and Q-AMSGradR perform competitively or even better than DQN in experiment?}
\end{list}

This paper addresses the above theoretical and experimental questions with affirmative answers.

\subsection{Main Contributions}

{\bf Theoretically}, we show that with linear function approximation (which is almost the only structure that the current tools for analysis of Q-learning can handle), both Q-AMSGrad and Q-AMSGradR converge to the global optimal solution under standard assumptions for Q-learning. To the best of our knowledge, this is the first non-asymptotic convergence guarantee on Q-learning that incorporates Adam-type update. Furthermore, a slight adaptation of our proof provides the convergence rate for the AMSGrad for conventional strongly convex optimization which has not been studied before and can be of independent interest.

{\bf Experimentally,} we demonstrate that the practical versions of Q-AMSGrad and Q-AMSGradR (referred to as Q-Adam and Q-AdamR) exhibit appealing experimental performance. 
In a batch of 23 Atari 2600 games, our experiments show that both Q-Adam and Q-AdamR outperform DQN by $50\%$ on average. Furthermore, Q-AdamR effectively reduces the performance variance and achieves a much more stable learning process. In our experiments for the linear quadratic regulator (LQR) problems, Q-AdamR converges even faster than the model-based value iteration (VI) solution. This is a rather surprising result given that the model-based VI has been treated as the performance upper bound for the Q-learning (including DQN) algorithms with target update~\cite{lewis2009reinforcement,yang2019theoretical}.

Detailed proofs and more experimental results will be available in the extended version of this paper on arXiv.org after the official publication of IJCAI Proceedings.

\subsection{Related Work}\label{sec:related_work}
We briefly review the related work as follows.

\vspace{1mm}
\paragraph{Theoretical analysis of Q-learning.} Since proposed in~\cite{watkins1992q}, the convergence of Q-learning has been extensively studied, particularly for the case with linear function approximation such as~\cite{bertsekas1996neuro,zou2019finite,Chen2019finiteQ,du2019provably}, to name a few, and more recently for the case with neural networks in~\cite{xu2019deepQ}, where the analysis exploits the approximate linear structure of neural networks in the overparamterized regime. All these existing analysis of Q-learning considers the vanilla SGD update, whereas our study is the first to analyze the more involved case with Adam-type updates.

\vspace{1mm}
\paragraph{Convergence analysis of Adam-type algorithms in conventional optimization.} Adam was proposed in~\cite{kingma2014adam} for speeding up the training of deep neural networks, and the regret bounds were characterized for Adam/AMSGrad in \cite{kingma2014adam,reddi2019convergence,tran2019convergence} for online convex optimization.  
Recently, convergence analysis of Adam/AMSGrad was provided for nonconvex optimization  in~\cite{zou2018sufficient,zhou2018convergence,chen2018convergence} and policy gradient~\cite{xiong2020amsgradRL}, in which such Adam-type algorithms were guaranteed to converge to a stationary point. To the best of our knowledge, our study provides the first convergence analysis of the Adam-type algorithms for Q-learning.

\vspace{1mm}
\paragraph{Empirical performance of Q-learning.} DQN learning and its improved variants of dueling network structure~\cite{wang2015dueling}, double Q-learning~\cite{van2016deep} and variance exploration and sampling schemes~\cite{schaul2015prioritized} have achieved significant success due to their superb performance  in practice. In contrast to such nested-loop algorithms (which involves the fitting of a target Q-function periodically), the Q-learning algorithms that strictly follow the alternating updates are much less explored in practice.~\cite{mnih2016asynchronous} proposed the asynchronous alternating Q-learning with competitive performance against DQN. However, the algorithm still relies on a slowly moving target network similar to DQN, and the multi-thread learning also complicates the computational setup.~\cite{tyler2018nd} studied the problem of value overestimation and proposed the non-delusional Q-learning algorithm that employs the so-called pre-conditioned Q-networks, which is also computationally complex.~\cite{knight2018natural} proposed a natural gradient propagation to improve the performance, where the gradient implementation is complex. Our experiments in this paper demonstrate that simple alternating Q-learning algorithms Q-AMSGrad and Q-AMSGradR without the complex designs as in \cite{mnih2016asynchronous,tyler2018nd,knight2018natural} have competitive and sometimes better performance than DQN. 

\noindent\textbf{Notations} We use $\|x\|:=\|x\|_{2}$ to denote the $\ell_2$ norm of a vector $x$, and use $\|x\|_{\infty}$ to denote the infinity norm. When $x,y$ are both vectors, $x/y, xy, x^2, \sqrt{x}$ are all calculated in the element-wise manner, which will be used in the update of Adam and AMSGrad. We denote $[n]=1,2,\dots,n$, and $\lfloor x\rfloor\in \mathbb{Z} $ as the integer such that $\lfloor x\rfloor \leq x < \lfloor x\rfloor+1$.


\section{Preliminaries} \label{sec:preliminary}
We consider a Markov decision process with a considerably large or continuous state space $\mathcal{S}\subset \mathbb{R}^{M}$ and action space $\mathcal{A} \subset \mathbb{R}^N$, a non-negative bounded reward function $R: \mathcal{S} \times \mathcal{A} \rightarrow [0, R_{\text{max}}]$, and a transition kernel $P(s'|s,a)$ that indicates the probability from a state-action pair $(s,a)$ to a state $s'$. We define $U(s)\subset \mathcal{A}$ as the admissible set of actions at state $s$, and $\pi:\mathcal{S} \rightarrow \mathcal{A}$ as a feasible stationary policy. We seek to solve a discrete-time sequential decision problem as follows:
\begin{align}
    \label{eq:optimal-control-problem}
    \underset{\pi}{\text{maximize}}
    & \ \ J_{\pi}(s_0) = \mathbb{E}_P\left[\sum_{t=0}^{\infty} \gamma^t R(s_t, \pi(s_t))\right], \nonumber\\
    \text{subject to} 
    & \ \ s_{t+1} \sim P(\cdot|s_t, a_t),
\end{align}
where $\gamma\in(0,1)$ is the discount factor. Let $J^{\star}(s):=J_{\pi^{\star}}(s)$ be the optimal value function when applying the optimal policy $\pi^{\star}$. The corresponding optimal Q-function can be defined as 
\begin{equation}
    \label{eq:optQ}
    Q^{\star}(s, a):= R(s, a) + \gamma\mathbb{E}_P J^{\star}(s'),
\end{equation}
where $s'\sim P(\cdot|s, a)$ and we use the same notation hereafter when no confusion arises.
In other words, $Q^{\star}(s, a)$ represents the reward of an agent who starts from state $s$ and takes action $a$ at the first step and then follows the optimal policy $\pi^{\star}$ thereafter.

\subsection{Q-learning Algorithm}
This paper focuses on the Q-learning algorithm that uses a parametric function $\hat{Q}(s, a ; \theta)$ to approximate the Q-function with a parameter $\theta$ having finite and relatively small dimensions. The update rule of Q-learning is given by
\begin{align}
 & T\hat Q(s,a;\theta_t) = R(s, a) + \gamma \underset{a'\in U(s')}{\max} \hat{Q}(s', a';\theta_{t});\label{eq:QtargetUpdate}\\
 & \theta_{t+1}  = \theta_t - \alpha_t \left( \hat{Q}_{t}(s, a; \theta_{t}) - T\hat Q(s,a;\theta_t) \right) \hat g_t,\label{eq:QthetaUpdate} \\
 & \hat g_t:= \hat g(\theta_t;s,a) = \frac{\partial}{\partial \theta_t}\hat{Q}_{t}(s, a; \theta_{t}), \label{eq:gt}
\end{align}
where $\alpha_t$ is the step size at time $t$.
It is clear that Q-learning performs the update by taking one step of temporal target update and one step of parameter learning in an alternating fashion.

\subsection{Linear Function Approximation}
Like most of the related work, we focus on the convergence analysis under the linear function approximation. 
A linear approximation of the Q-function $\hat{Q}(s, a;\theta)$ can be written as
\begin{equation}
    \label{eq:linearApprox}
    \hat{Q}(s, a;\theta) = \phi(s, a)^T\theta, 
\end{equation}
where $\theta\in\mathbb{R}^d$, and $\phi:\mathcal{S}\times \mathcal{A}\rightarrow \mathbb{R}^d$ is a vector function of size $d$, and the elements of $\Phi$ represent the nonlinear kernel (feature) functions.

\begin{remark}\label{rk:nn}
We note that recent work \cite{xu2019deepQ} established the convergence rate of Q-learning with neural network approximation, which exploits the approximate linear structure of the neural network in the overparameterized regime. Thus, our analysis under the linear function approximation can be generalized to the function class of overparameterized neural networks by applying the techniques developed in recent work \cite{xu2019deepQ}.
\end{remark}

\section{Convergence Analysis of Q-AMSGrad} \label{sec:stability}
In this section, we characterize the convergence guarantee for Q-learning under Adam-type updates.

\subsection{Q-AMSGrad Algorithm}
Although Adam has obtained great success as an optimizer in deep learning, it is well known that Adam by nature is non-convergent even for simple convex loss functions~\cite{reddi2019convergence}. Instead, a slightly modified version called AMSGrad~\cite{reddi2019convergence} is widely used to study the convergence property of Adam-type algorithms in conventional optimization. Here, we apply the update rule of AMSGrad to the Q-learning algorithm and refer to such an algorithm as Q-AMSGrad. Algorithm \ref{alg:opql-AMSGrad} describes Q-AMSGrad in detail.

More specifically, the iterations of Q-AMSGrad evolve by updating the exponentially decaying average of historical gradients ($m_t$) and squared historical gradients ($v_t$). The hyper-parameters $\beta_1,\beta_2$ are used to exponentially decrease the rate of the moving averages. The difference between AMSGrad and Adam lies in the fact that AMSGrad makes the sequence $\hat v_{t,i}$ increasing along the time step $t$ for each entry $i\in [d]$, whereas Adam does not guarantee such a property.
\begin{algorithm}
 	\caption{Q-AMSGrad} \label{alg:opql-AMSGrad} 
 	\begin{algorithmic}[1]
 		\STATE 	{\bf Input:}   $\alpha, \lambda, \theta_{1}, \beta_1, \beta_2, m_0 = 0, \hat v_0 = 0$.
		\FOR{ $t=1, 2, \ldots, T $}
		\STATE $\alpha_t = \frac{\alpha}{\sqrt{t}}$, $\beta_{1t} = \beta_1 \lambda^t$
		\STATE Observe data $(s_t, a_t, s_{t+1}) $ from policy $\pi$ and transition probability $P$
		\STATE $b_t = R(s_t,a_t) + \gamma\underset{a'}{\max}\phi^T(s_{t+1},a')\theta_t$
		\STATE $g_t = \left(\phi^T(s_t, a_t)\theta_t - b_t\right)\phi(s_t, a_t)$
		\STATE $m_{t} = (1-\beta_{1t})m_{t-1} + \beta_{1t} g_t$
		\STATE $v_{t} = (1-\beta_2)\hat v_{t-1} + \beta_2 g_t^2$
		\STATE $\hat v_t = \max (\hat{v}_{t-1}, v_t),\ \hat{V}_t = diag(\hat{v}_1,\dots,\hat{v}_d) $
		\STATE $\theta_{t+1} = \Pi_{\mathcal{D}, \hat V_t^{1/4}} ( \theta_t - \alpha_t \hat{V}_t^{-\frac{1}{2}}m_t)$\\ where $\Pi_{\mathcal{D}, \hat V_t^{1/4}}(\theta') = \underset{\theta\in\mathcal{D}}{\min}\norm{\hat V_t^{1/4}\left(\theta' - \theta\right)}$. 
		\ENDFOR
        \STATE 	{ {\bf Output:} $\frac{1}{T}\sum_{t=1}^T \theta_t $}
 	\end{algorithmic}
\end{algorithm} 

\subsection{Convergence Result}
Before stating the main theorem, we first introduce some technical assumptions and lemmas for our analysis.
\begin{assumption}\label{asp:boundPhi}
For any state-action pair $(s,a)\in\mathcal{S}\times\mathcal{A}$, the kernel function $\phi$ is uniformly bounded and we have
\begin{equation}
    \norm{\phi(s,a)} \leq 1,\quad \forall (s,a).
\end{equation}
\end{assumption}
This assumption is mild since we can normalize the kernel function if the kernel function is uniformly bounded. It is widely applied in the literature to simplify the analysis of RL algorithms with linear function approximation \cite{bhandari2018finite,Chen2019finiteQ}.

\begin{assumption}\cite[Lemma 6.7]{Chen2019finiteQ}
    \label{asp:qlearning}
    At each iteration $t$, the noisy gradient is unbiased, i.e. $ g_t = \bar g_t + \xi_t$ with $\mathbb{E}\xi_t = 0$ where $\bar g_t=\mathbb{E}[g_t]$. The equation $\bar g(\theta) = 0$ has a unique solution $\theta^\star$, and there exists a $c>0$, such that for any $\theta\in\mathbb{R}^d$ we have
    \begin{equation}\label{eq:qasp}
        (\theta - \theta^\star)^T\bar g(\theta)\geq c\norm{\theta - \theta^\star}^2.
    \end{equation}
\end{assumption}
Assumption~\ref{asp:qlearning} has been proved as a key technical lemma in~\cite{Chen2019finiteQ} under certain assumptions, which appears to be the weakest among the existing studies for establishing the convergence guarantee for Q-learning with linear function approximation. It is the standard assumption in the related literature to analyze the convergence Q-learning with linear function approximation \cite{zou2019finite,Chen2019finiteQ,xu2019deepQ}.

\begin{assumption}
    \label{asp:boundedDomain}
    The domain $\mathcal{D}\subset \mathbb{R}^d$ of approximation parameters is a ball originating at $\theta = 0$ with bounded diameter containing $\theta^\star$. That is, there exists $D_\infty$, such that $\norm{\theta_m - \theta_n} < D_\infty, \forall \theta_m, \theta_n\in \mathcal{D}$, and $\theta^\star\in\mathcal{D}$.
\end{assumption}
This Assumption can be easily satisfied when we apply a projected algorithm, and is standard in the theoretical analysis of Adam-type algorithms~\cite{chen2018convergence,zhou2018convergence}.

Based on the above assumptions, we can immediately obtain the bounded property of the gradient, which is stated in the following lemma.
\begin{lemma}\label{lem:mvbound}
Under Assumptions \ref{asp:boundPhi} and \ref{asp:boundedDomain}, at each iteration $t$, the gradient estimator $g_t$ in Q-AMSGrad is uniformly bounded. That is,
\begin{equation}
    \norm{g_t}_\infty \leq \norm{g_t} \leq R_{\max} + (1+\gamma) D_{\infty}, \quad \forall t.
\end{equation}
In addition, we denote $G_\infty=R_{\max} + (1+\gamma) D_{\infty}$ and let $\{ m_t, \hat{v}_t\}$ for $t=1,2,\dots$ be sequences generated by Algorithm \ref{alg:Q-AMSGradR}. Then we have 
$$\norm{\bar g_t}\leq G_\infty, \norm{m_t} \leq G_\infty, \norm{\hat{v}_t} \leq G_\infty^2.$$
\end{lemma}

We next provide the non-asymptotic convergence of Q-AMSGrad 
in the following theorem.
\begin{theorem}
    \label{thm:adamStability}\label{eq:adamThm}
    (Convergence of Q-AMSGrad)
    Suppose $\alpha_t=\frac{\alpha}{\sqrt{t}}, \beta_{1t}=\beta_1\lambda^t$ and $\delta=\beta_1/\beta_2$ with $\delta,\lambda\in(0,1)$ for $t=1,2,\dots$ in Algorithm \ref{alg:opql-AMSGrad}. Given Assumptions \ref{asp:boundPhi} $\sim$ \ref{asp:boundedDomain}, the output of Q-AMSGrad satisfies:
    \begin{equation}\label{eq:thm1}
        \begin{aligned}
        \mathbb{E} & \norm{\theta_{out}- \theta^\star} \\ 
        & \leq \frac{B_1}{T} + \frac{B_2}{\sqrt{T}} + \frac{B_3\sqrt{1+\log T}}{T}\sum_{i=1}^d\mathbb{E}\norm{g_{1:T,i}},
        \end{aligned}
    \end{equation}
    where $B_1=\frac{G_\infty D_\infty^2}{2\alpha_2 c (1-\beta_1)} + \frac{\beta_1 G_\infty D_\infty^2 }{2\alpha c(1-\beta_1)(1-\lambda)^2} + \norm{\theta_1-\theta^\star}^2$, $ B_2=\frac{d G_\infty D_\infty^2}{2\alpha c(1-\beta_1)}$, and $B_3=\frac{\alpha  (1+\beta_1)}{2c(1-\beta_1)^2(1-\delta)\sqrt{1-\beta_2}}$.
    
\end{theorem}

In Theorem \ref{thm:adamStability}, $B_1, B_2, B_3 $ in the bound in~\Cref{eq:thm1} are constants and independent of time. Therefore, under the choice of the stepsize and hyper-parameters in Algorithm \ref{alg:opql-AMSGrad}, Q-AMSGrad achieves a convergence rate of $\mathcal{O}\left( \frac{1}{\sqrt{T}} \right)$ when $\sum_{i=1}^d\norm{g_{1:T,i}} << \sqrt{T}$ \cite{reddi2019convergence}.
\begin{remark}~\label{rmk: ql and opt}
Our proof of convergence here has two major differences from that for AMSGrad in~\cite{reddi2019convergence} in conventional optimization: (a) The two algorithms are quite different. Q-AMSGrad is a Q-learning algorithm alternatively finding the best policy with a moving target, whereas AMSGrad is an optimizer for conventional optimization and does not have alternating nature. (b) Our analysis is on the convergence rate whereas~\cite{reddi2019convergence} provides regret bound. In fact, a slight modification of our proof also provides the convergence rate of AMSGrad for conventional strongly convex optimization, which can be of independent interest. Moreover, our proof avoids the theoretical error in the proof in~\cite{reddi2019convergence} as pointed out by~\cite{tran2019convergence}.
\end{remark}

\section{Convergence Analysis of Q-AMSGradR}

In this section, we propose to incorporate a momentum restart technique to Q-AMSGrad in order to improve its performance. We first introduce the algorithm and then provide the convergence analysis for such an algorithm. We demonstrate its desired experimental performance in Section \ref{sec:experiments}.

\subsection{Q-AMSGrad Algorithm with Momentum Restart}

We introduce the restart technique to Q-AMSGrad and propose Q-AMSGradR as shown in Algorithm~\ref{alg:Q-AMSGradR}.
Q-AMSGradR applies the same update rule as Algorithm \ref{alg:opql-AMSGrad}, but periodically resets $m_t, \hat{v}_t$ with a period of $r$, i.e., $m_t = 0, \hat v_t = 0, \forall t=kr,k = 1,2,\cdots$. We explain such an idea further as follows.

\begin{algorithm}[t]
 	\caption{Q-AMSGradR} \label{alg:Q-AMSGradR} 
 	\begin{algorithmic}[1]
 		\STATE 	{\bf Input:}   $\alpha, \lambda, \theta_{1}, \beta_1, \beta_2, m_0 = 0, \hat v_0 = 0$.
		\FOR{ $t=1, 2, \ldots, T $}
		\IF { $\textrm{mod}(t,r) = 0$ }
		\STATE $m_t=0$, $\hat v_t=0$
		\ENDIF
		\STATE $\alpha_t = \frac{\alpha}{\sqrt{t}}$, $\beta_{1t} = \beta_1 \lambda^t$
		\STATE Observe data $(s_t, a_t, s_{t+1}) $ from policy $\pi$ and transition probability $P$
		\STATE $b_t = R(s_t,a_t) + \gamma\underset{a'}{\max}\phi^T(s_{t+1},a')\theta_t$
		\STATE $g_t = \left(\phi^T(s_t, a_t)\theta_t - b_t\right)\phi(s_t, a_t)$
		\STATE $m_{t} = (1-\beta_{1t})m_{t-1} + \beta_{1t} g_t$
		\STATE $v_{t} = (1-\beta_2)\hat v_{t-1} + \beta_2 g_t^2$
		\STATE $\hat v_t = \max (\hat{v}_{t-1}, v_t),\ \hat{V}_t = diag(\hat{v}_1,\dots,\hat{v}_d) $
		\STATE $\theta_{t+1} = \Pi_{\mathcal{D}, \hat V_t^{1/4}} \left( \theta_t - \alpha_t \hat{V}_t^{-\frac{1}{2}}m_t\right)$
		\ENDFOR
        \STATE 	{ {\bf Output:} $\frac{1}{T}\sum_{t=1}^T \theta_t $}
 	\end{algorithmic}
\end{algorithm} 

Traditional momentum-based algorithms largely depend on the historical gradient direction. When part of the historical information is incorrect, the estimation error tends to accumulate. The restart technique can be employed to deal with such an issue. One way to restart the momentum-based methods is to initialize the momentum at some restart iteration. That is, at restart iteration $r$, we reset $m_r,v_r$, i.e., $m_{r} = 0, v_r =0$, which yields $\theta_{r+1}=\theta_r$.
It is an intuitive implementation technique to adjust the trajectory from time to time, and can usually help mitigate the aforementioned problem while keeping fast convergence property. For the implementation, we execute the restart periodically with a period $r$. It turns out that the restart technique can significantly improve the numerical performance, which can be seen in Section~\ref{sec:experiments}.

\subsection{Convergence Result}

In the following theorem, we provide the non-asymptotic convergence for Q-AMSGradR.
\begin{theorem}
    \label{thm:adamRestartStability}
    (Convergence of Q-AMSGradR) Suppose $\alpha_t=\frac{\alpha}{\sqrt{t}}, \beta_{1t}=\beta_1\lambda^t$ and $\delta=\beta_1/\beta_2$ with $\delta,\lambda\in(0,1)$ for $t=1,2,\dots$ in Algorithm \ref{alg:Q-AMSGradR}. Given Assumptions \ref{asp:boundPhi} $\sim$ \ref{asp:boundedDomain}, the output of Q-AMSGradR satisfies: 
    \begin{equation}\label{eq:thm2}
    \begin{aligned}
    & \mathbb{E}\norm{\theta_{out}-\theta^\star} \\
    & \leq \frac{B_1}{T} + \frac{B_2\sqrt{1+\log T}}{T}\sum_{i=1}^d\mathbb{E}\norm{g_{1:T,i}} \\
    & + \frac{B_3}{T}\left(\sqrt{T} + \sum_{k=1}^{\lfloor T/r \rfloor }\sqrt{kr-1} \right)\\ 
    & +\frac{1}{T} \sum_{k=0}^{\lfloor T/r \rfloor }\left( \frac{G_\infty D_\infty^2}{\alpha} \sqrt{kr+2}  + B_4 \mathbb{E}\norm{\theta_{kr} - \theta^\star}^2  \right),
    \end{aligned}
    \end{equation}
    where $B_1 =\frac{\beta_1D_\infty^2 G_\infty}{2\alpha c(1-\beta_1)(1-\lambda)^2}$, $B_2 = \frac{\alpha  (1+\beta_1)}{2c(1-\beta_1)^2(1-\delta)\sqrt{1-\beta_2}}$, $B_3=\frac{d G_\infty D_\infty^2}{2\alpha c(1-\beta_1)}$, and $B_4=4c(1-\beta_1)$.
\end{theorem}
Theorem~\ref{thm:adamRestartStability} indicates that for Q-AMSGradR to enjoy a convergence rate of $\mathcal{O}\left( \frac{1}{\sqrt{T}} \right)$, the restart period $r$ needs to be sufficiently large and $\sum_{i=1}^d\norm{g_{1:T,i}} << \sqrt{T}$. In practice as demonstrated by the experiments in Section~\ref{sec:experiments}, Q-AMSGradR typically performs well, not necessarily under the theoretical conditions.

\section{Experimental Performance}\label{sec:experiments}
In this section, we empirically evaluate the Q-learning algorithms studied in this paper. We first study the linear quadratic regulator (LQR) problem, which serves as a direct numerical demonstration of the convergence analysis under linear function approximation. We then use the Atari 2600 games~\cite{brockman2016openai}, a classic benchmark for DQN evaluations, to demonstrate the effectiveness of the Q-learning algorithms for complicated tasks. Note that our experiments test the performance of Q-Adam and Q-AdamR, which serve as practical versions of Q-AMSGrad and Q-AMSGradR by adopting Adam for Q-learning and keeping all major properties of Q-AMSGrad and Q-AMSGradR including the alternating update between Q-function update and parameter fitting. 

Our main focus here lies in: (a) comparison between vanilla Q-Adam and that with momemtum restart (Q-AdamR), (b) comparison between Q-Adam/Q-AdamR and vanilla Q-learning (through LQR), and (c) comparison between Q-Adam/Q-AdamR and DQN (through Atari 2600 games). 
Both of our experiments show that Q-AdamR outperforms Q-Adam, vanilla Q-learning and DQN in terms of convergence speed and variance reduction. Compared with DQN in the empirical experiments of Atari games, under the same hyper-parameter settings, Q-Adam and Q-AdamR improve the performance of DQN by $50\%$ on average.

\subsection{DQN Algorithm}\label{sec:DQN}
As DQN is also included in this work for performance comparison. We recall the update of DQN in the following as reference.
Differently from the vanilla Q-learning, DQN updates the parameters in a nested loop. Within the $t$-th inner loop, DQN first obtains the target Q-function as in~\Cref{eq:DQNtargetUpdate}, and then uses a neural network to fit the target Q-function by running $Y$ steps of a certain optimization algorithm as~\Cref{eq:DQNsuperLearning}. The update rule of DQN is given as follows.
\begin{align}
    & T\hat Q(s,a;\theta_t^0) = R(s, a) + \gamma \underset{a'\in U(s')}{\max} \hat{Q}(s', a';\theta_{t}^0),\label{eq:DQNtargetUpdate}\\
    & \theta_{t}^{Y} = Optimizer(\theta_t^0, T\hat Q(s,a;\theta_t^0)),\label{eq:DQNsuperLearning}
\end{align}
where $Optimizer$ can be SGD or Adam for example, and~\Cref{eq:DQNsuperLearning} is thus a supervised learning process with $T\hat Q(s,a;\theta_t^0))$ as the "supervisor".
At the $t$-th outer loop, DQN performs the so-called target update as
\begin{equation}
    \theta_{t+1}^0 = (1-\tau)\theta_{t}^0 + \tau\theta_t^{Y}.
\end{equation}
In practice, when one of the momentum-based optimizers is adopted for~\Cref{eq:DQNsuperLearning}, such as Adam, it is only initialized once at the beginning of the first inner loop. The historical gradient terms then accumulate throughout multiple inner loops with different targets. 

Generally speaking, the difference between Q-Adam/Q-AdamR and DQN mainly lies in Q-Adam/Q-AdamR takes one-step Q-function update and one-step model parameter fitting alternatively, whereas DQN takes one-step Q-function update followed by a sufficient large number of steps for model parameter fitting (towards a target Q-function). 

\subsection{Linear Quadratic Regulator}\label{subsec:LQR}
We numerically validate the performance of Q-Adam and Q-AdamR through an infinite-horizon discrete-time LQR problem. A typical model-based solution (with known dynamics), known as the discrete-time algebraic Riccati equation (DARE), is adopted to derive the optimal policy $u_t^{\star} = -K^{\star}x_t$. The performance of the learning algorithm is then evaluated at each step of iterate $t$ with the Euclidean norm $\norm{K_t - K^{\star}}$. Given the problem nature of LQR, we also re-scale the loss term of $L(\theta_t) := \hat{Q}_{t}(s, a; \theta_{t}) - T\hat Q(s,a;\theta_t)$ in~\Cref{eq:QthetaUpdate} as $\tilde{L}(\theta_t) = \tilde{\tau}^2L(\theta_t)$ with some scaling factor $\tilde{\tau} \in (0,1]$, which is beneficial for stabilizing the learning process. The performance result for each method is averaged over 10 trials with different random seeds. All algorithms share the same set of random seeds and are initialized with the same $\theta_0$. The hyper-parameters of the learning settings are also consistent and further details are shown in Table~\ref{tbl:lqr}. Note that for all the implementations, we also adopt the double Q-update~\cite{van2016deep} to help prevent over-estimations of the Q-value. The performance results are provided in Figure~\ref{fig: lqr}. Here we highlight the main observations from the LQR experiments.

\begin{table}[t]
\centering
\begin{tabular}{c c c c c c c c }
\toprule
Step size  & $\tilde{\tau}$  & Adam $\beta_1$ & Adam $\beta_2$  \\ 
0.0001    & 0.01            & 0.9            & 0.999 \\ \hline
Restart period $r$ & \multicolumn{2}{c}{Stop criterion}  & $\gamma$ \\
100                & \multicolumn{2}{c}{$\norm{K_i - K^{\star}}_2 \leq 10^{-4}$} & 1 \\
\bottomrule
\end{tabular}
\caption{Hyper-parameters for LQR experiments.}
\label{tbl:lqr}
\end{table}

\paragraph{Q-AdamR outperforms DARE.} In ideal cases where data sampling perfectly emulates the system dynamics and the target is accurately learned in each inner loop, DARE for LQR would become equivalent to the DQN-like update if the neural network is replaced with a parameterized linear function. In practice, such ideal conditions are difficult to satisfy, and hence the actual Q-learning with target update is usually far slower (in terms of the number of steps of target updates) than DARE. Note that Q-AdamR performs significantly well and even converges faster than DARE, and thus implies it is faster than the most well-performing Q-learning with target update.

\paragraph{Q-AdamR outperforms Q-Adam.} Overall, under the same batch sampling scheme and restart period, Q-AdamR achieves a faster convergence and smaller variance than Q-Adam. 

\begin{figure}[t]
\begin{minipage}{.475\textwidth}
    \centering
    \includegraphics[width=\textwidth]{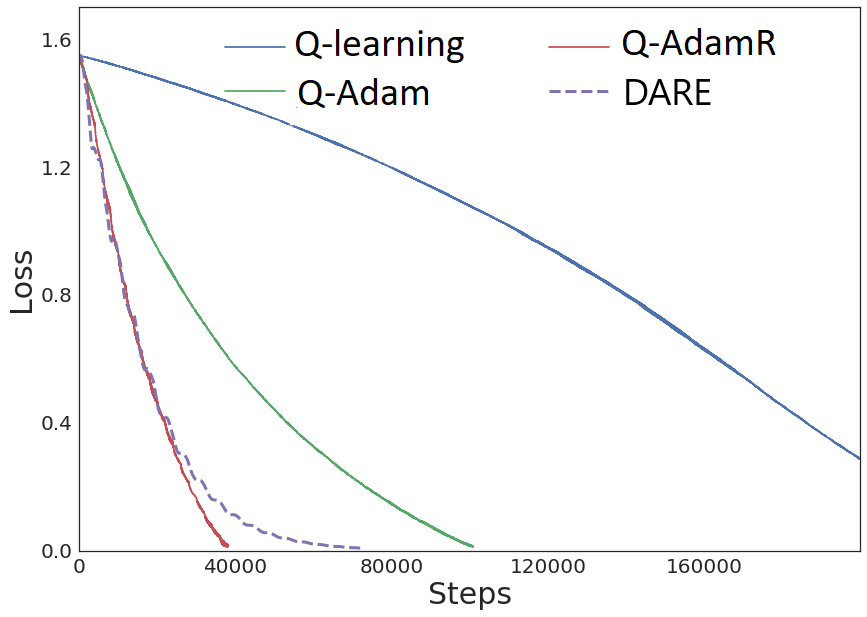}
    \caption{LQR experiments with performance evaluated in terms of policy loss $\norm{K_t - K^{\star}}_2$.}
    \label{fig: lqr}
\end{minipage}
\vspace{1em}

\begin{minipage}{.475\textwidth}
    \centering
    \includegraphics[width=\textwidth]{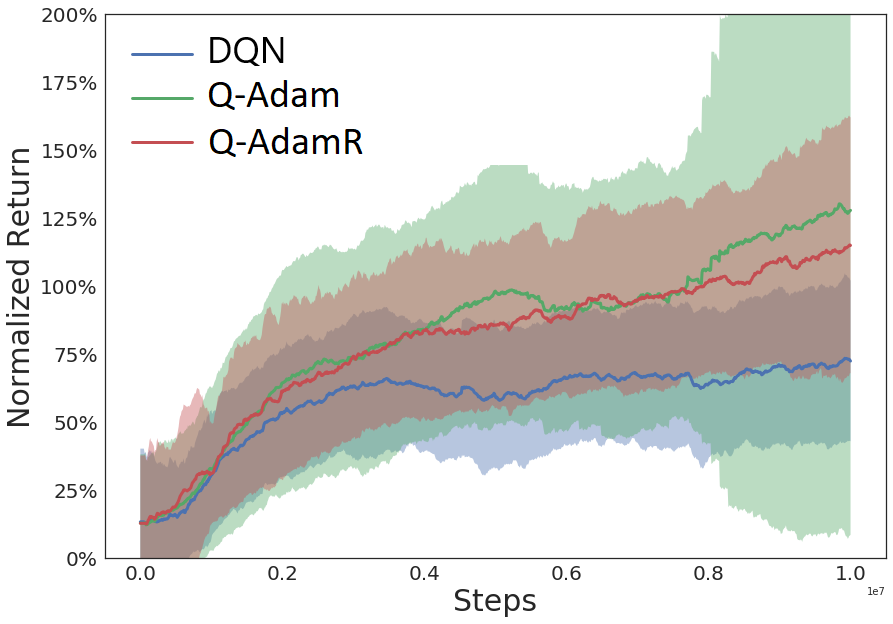}
    \caption{Atari game experiment with performance normalized and averaged over 23 games.}
    \label{fig: dqn}
\end{minipage}
\end{figure}

\subsection{Atari Games}
We apply the Q-Adam and Q-AdamR algorithms to the more challenging tasks of deep convolutional neural network playing a group of Atari 2600 games. The particular DQN we train to compare against adopts the dueling network structure~\cite{wang2015dueling}, double Q-learning setup~\cite{van2016deep}, $\epsilon$-greedy exploration and experience replay~\cite{mnih2015human}. Adam is also adopted, without momentum restart, as the optimizer for the inner-loop supervised learning process. Q-Adam and Q-AdamR are implemented using the identical setup of network construction, exploration and sampling strategies.

We test all three algorithms with a batch of 23 Atari games. The choice of 10 million steps of iteration is a common setup for benchmark experiments with Atari games. Although this does not guarantee the best performance in comparison with more time-consuming training with 50 million steps or more, it is sufficient to illustrate different performances among the selected methods. The software infrastructure is based on the baseline implementation of OpenAI. Selections of the hyper-parameters are listed in Table~\ref{tbl:dqn}. We summarize the results in Figure~\ref{fig: dqn}. The overall performance is illustrated by first normalizing the return of each method with respect to the results obtained from DQN, and then averaging the performance of all 23 games to obtain the mean return and standard deviation. Considering we use a smaller buffer size than common practice, DQN is not consistently showing improved return over all tested games. Therefore, the self-normalized average return of DQN in Figure~\ref{fig: dqn} is not strictly increasing from $0$ to $100\%$. 

Overall, both Q-Adam and Q-AdamR achieve significant improvement in comparison with the DQN results. In particular, AltQ-Adam increases the performance by over $100\%$ in some of the tasks including Asterix, BeamRider, Enduro, Gopher, etc. However, it also illustrates certain instability with complete failure on Amidar and Assault. This is also capture by the higher variance illustrated on Figure~\ref{fig: dqn}. Periodic restart (Q-AdamR) resolves this issue efficiently with an on-par performance on average and far smaller variance. In terms of the maximum average return, Q-Adam and Q-AdamR perform no worse then DQN on 17 and 20 games respectively out of the 23 games being evaluated. Furthermore, if we consider having a final score that is smaller or equal to the start score as a learning failure, DQN fails the learning in 5 out of 23 games (Asteroids, DoubleDunk, Gravitar, Pitfall, Tennis), Q-Adam fails in 3 out of 23 games (Amidar, Assault, Asteroids) and Q-AdamR does not fail in any of the tasks. That is, Q-AdamR not only reduces the variance, but also provides a more consistent performance across the task domain. This implies that momentum restart effectively corrects the accumulated error and stabilizes the training process.  

\begin{table}[t]
\centering
\begin{tabular}{c c c c }
\toprule
Step size           & Scale factor $\tilde{\tau}$  & Adam $\beta_1$            & Adam $\beta_2$  \\  
0.0001              & 0.0001                       & 0.9                       & 0.999     \\ \hline 
$r$  & Buffer size                  & $\gamma$                  & Batch size $B$     \\  
$10^4$              & $10^5$                       & 0.99                      & 32                           \\ \hline
\multicolumn{2}{c}{Total training steps $K$} & \multicolumn{2}{c}{Target update frequency}    \\
\multicolumn{2}{c}{$10^7$ } & \multicolumn{2}{c}{$10^4$}                     \\
\bottomrule
\end{tabular}
\caption{Hyper-parameters for Atari games experiments of DQN, Q-Adam and Q-AdamR.}
\label{tbl:dqn}
\end{table}

\section{Conclusion}\label{sec:conclusion}
We study two Q-learning algorithms with Adam-type updates, and demonstrate their superior performance over the vanilla Q-learning and DQN algorithms through a linear quadratic regulator problem and a batch of 23 Atari 2600 games. 

It is of considerable future interest to further investigate the potential of the restart scheme. One possible direction is to develop an adaptive restart mechanism with changing period determined by an appropriately defined signal of restart. This will potentially relieve the effort in hyper-parameter tuning of finding a good fixed period. 

\section*{Acknowledgements}
The work was supported in part by the U.S. National Science Foundation under Grants CCF-1761506, ECCS-1818904, CCF-1909291 and CCF-1900145, and the startup funding of the Southern University of Science and Technology (SUSTech), China.

\newpage

\bibliographystyle{named}
\bibliography{ijcai20}

\newpage
\onecolumn
\appendix

\noindent {\Large \textbf{Supplementary Material}}

\section{Further Details and Results on Experiments}
We discuss more details on the experiment setup and provide further results that are not included in Section \ref{sec:experiments}.

\subsection{Linear Quadratic Regulator}\label{subsec:lqrApp}
The linear quadratic regulator (LQR) problem is of great interest for control community where Lewis et al. applies PQL to both discrete-time problems~\cite{lewis2009reinforcement} and continuous-time problems~\cite{vamvoudakis2017q,vrabie2009adaptive}. 

We empirically validate the proposed algorithms through an infinite-horizon discrete-time LQR problem defined as
\begin{align}
    \label{eq:lqr-problem}
    & \underset{\pi}{\text{minimize}}
    & & J = \sum_{t=0}^{\infty}\left( x_t^T Q x_t + u_t^T R u_t + 2x_t^TNu_t\right), \nonumber\\
    & \text{subject to}
    & & x_{t+1} = Ax_t + Bu_t,\nonumber
\end{align}
where $u_t = \pi(x_t)$. 

A typical model-based solution (with known $A$ and $B$) considers the problem backwards in time and iterates a dynamic equation known as the discrete-time algebraic Riccati equation (DARE):
\begin{equation}
    \label{eq: DARE}
    P = A^TPA - (A^TPB + N)(R + B^TPB)^{-1}(B^TPA + N^T) + Q,
\end{equation}
with the cost-to-go $P$ being positive definite. The optimal policy satisfies $u_t^{\star} = -K^{\star}x_t$ with
\begin{align}
    \label{eq: optimal linear solution}
    K^{\star} = (R + B^TPB)^{-1} (N^T + B^TPA).
\end{align}

For experiments, we parameterize a quadratic Q-function with a matrix parameter $H$ in the form of
\begin{align}
    \label{eq: Q-function for lqr}
    Q(x, u ;H) = 
    \left[ \begin{array}{c} x \\u\end{array} \right]^T \left[ \begin{array}{c c} H_{xx} &H_{xu} \\ H_{ux} &H_{uu}\end{array} \right] \left[ \begin{array}{c} x \\u\end{array} \right].
\end{align}
The corresponding linear policy satisfies $u = -Kx$, and $K = H_{uu}^{-1}H_{ux}$. The performance of the learning algorithm is then evaluated at each step of iterate $i$ with the Euclidean norm $\norm{K_i - K^{\star}}_2$. 

\subsection{Atari Games}
We list detailed experiments of the 23 Atari games evaluated with the proposed algorithms in Figure~\ref{fig:dqn_all_exp}. All experiments are executed with the same set of two random seeds. Each task takes about 20-hour of wall-clock time on a GPU instance. All three methods being evaluated share similar training time. AltQ-Adam and AltQ-AdamR can be further accelerated in practice with a more memory-efficient implementation considering the target network is not required. We keep our implementation of proposed algorithms consistent with the DQN we are comparing against. Other techniques that are not included in this experiment are also compatible with AltQ-Adam and AltQ-AdamR, such like asynchronous exploration~\cite{mnih2013playing} and training with decorrelated loss~\cite{mavrin2019deep}.

Overall, AltQ-Adam significantly increases the performance by over $100\%$ in some of the tasks including Asterix, BeamRider, Enduro, Gopher, etc. However, it also illustrates certain instability with complete failure on Amidar and Assault. This is mostly caused by the sampling where we are using a relevantly small buffer size with $10\%$ of the common configured size in Atari games with experience replay. Notice that those failures tend to appear when the $\epsilon$-greedy exploration has evolved to a certain level where the immediate policy is effectively contributing to the accumulated experience. This potentially amplifies the biased exploration that essentially leads to the observed phenomenon. 

Interstingly, AltQ-AdamR that incorporates the restart scheme resolves the problem of high variance of average return brought by AltQ-Adam and provides a more consistent performance across the task domain. This implies that momentum restart effectively corrects the accumulated error and stabilizes the training process.

\begin{figure}[H]
    \centering
    \includegraphics[width=0.88\textwidth]{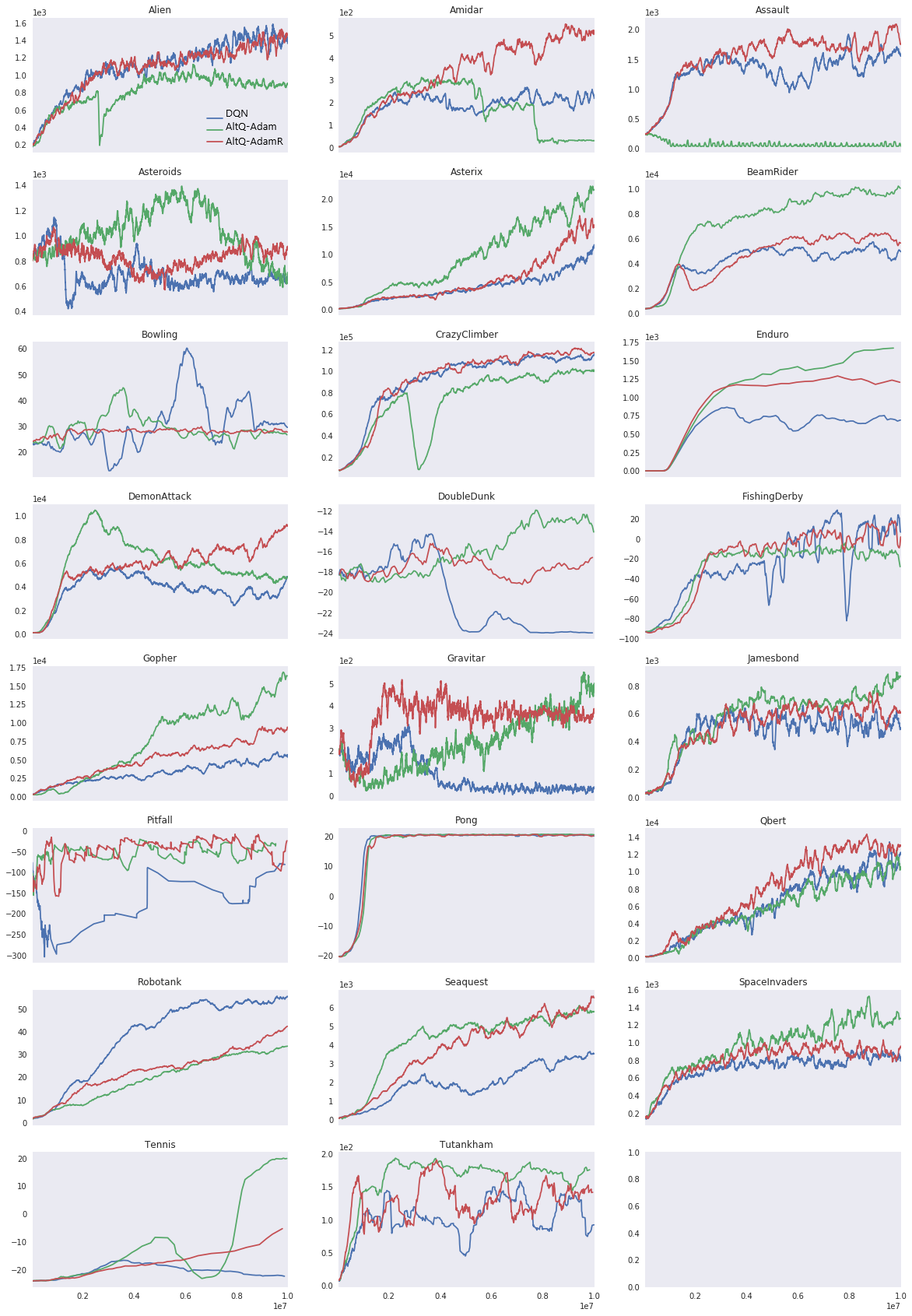}
    \caption{Experiment results of 23 Atari games with DQN, AltQ-Adam and AltQ-AdamR}
    \label{fig:dqn_all_exp}
\end{figure}

\begin{table}[H]
    \centering
    \begin{tabular}{c|c|c|c}
    \toprule
        Task & DQN & AltQ-Adam & AltQ-AdamR \\ \hline
        Alien       & 1529          & 1125           & \textbf{1587} \\
        Amidar      & 269           & 313            & \textbf{551}  \\
        Assault     & 1925          & 260            & \textbf{2097} \\
        Asteroids   & 1147          & \textbf{1394}  & 1069 \\
        Asterix     & 11794         & \textbf{22413} & 17064 \\
        BeamRider   & 5728          & \textbf{10210} & 6458  \\
        Bowling     & \textbf{60}   & 45             & 30    \\
        CrazyClimber &116422        & 102731         & \textbf{121770} \\
        Enduro      & 866           & \textbf{1671}  & 1291 \\
        DemonAttack & 5729          & \textbf{10485} & 9273 \\
        DoubleDunk  & -14           & \textbf{-12}   & -15 \\
        FishingDerby& \textbf{29.01}& -4             & 19  \\
        Gopher      &  6066         & \textbf{16863} & 9508 \\
        Gravitar    &  316          & \textbf{551}   & 518  \\
        Jamesbond   &  663          & \textbf{899}   & 756  \\
        Pitfall     &   -76         & -20            & \textbf{-7} \\
        Pong        &  20.68        & \textbf{20.79} & 20.74 \\
        Qbert       &  13453        & 12487          & \textbf{14352} \\
        Robotank    &  \textbf{56}  & 34             & 42 \\
        Seaquest    &  3652         & 6121           & \textbf{6624} \\
        Spaceinvaders& 923          & \textbf{1528}  & 1036 \\
        Tennis      & -17           & \textbf{20}    & -5 \\
        Tutankham   & 159           & \textbf{194}   & 191 \\
    \bottomrule
    \end{tabular}
    \caption{Best empirical return of 23 Atari games with DQN, AltQ-Adam and AltQ-AdamR}
    \label{tab:my_label}
\end{table}

\section{ Proof of Lemma \ref{lem:mvbound} }

The proof can be proceeded by the assumptions on the bounded kernel function and the bounded domains, which yields
\begin{align*}
    g_t  &= \left(\phi^T(s_t, a_t)\theta_t - R(s_t,a_t) - \gamma\underset{a'}{\max}\phi^T(s_{t+1},a')\theta_t\right)\phi(s_t, a_t)\\
    &\overset{\text{(i)}}{\leq}\norm{\phi^T(s_t, a_t)\theta_t - R(s_t,a_t) - \gamma\underset{a'}{\max}\phi^T(s_{t+1},a')\theta_t}\norm{\phi(s_t, a_t)}\\
    &\overset{\text{(ii)}}{\leq} \norm{\phi^T(s_t, a_t)\theta_t} + \norm{R(s_t,a_t)} + \gamma\norm{\underset{a'}{\max}\phi^T(s_{t+1},a')\theta_t}\\
    &\overset{\text{(iii)}}{\leq} \norm{g_t} \leq R_{\max} + (1+\gamma) D_{\infty},
\end{align*}
where (i) follows from Cauchy-Schwarz inequality, (ii) follows from the triangle inequality and Assumption \ref{asp:boundPhi} and (iii) follows from Assumptions \ref{asp:boundPhi} and \ref{asp:boundedDomain}.

Recall that 
\begin{align*}
    \bar g_t = \underset{\mu}{\mathbb{E}} \left[ \left(\phi^T(s_t, \pi(s_t))\theta_t - R(s_t,\pi(s_t)) - \gamma\underset{a'}{\max}\phi^T(s_{t+1},a')\theta_t\right)\phi(s_t, \pi(s_t)) \right],
\end{align*}
where $\mu$ is the stationary distribution of the states. Then the bound of $\bar g_t$ can be obtained by using the similar steps and techniques as those of bounding $g_t$.

The bounds of $m_t$ and $\hat v_t$ can be obtained by induction.
To this end, we first check that $m_0=0$ and $\norm{m_1}=\norm{\beta_{11}g_1}\leq G_\infty$. Assume that $\norm{m_{t-1}}\leq G_\infty$, then we have
\begin{align*}
    \norm{m_{t}} &= \norm{ (1-\beta_{1t})m_{t-1} + \beta_{1t} g_t }\\
    &\leq (1-\beta_{1t})\norm{m_{t-1}} + \beta_{1t} \norm{g_t}\\
    &\leq (1-\beta_{1t})G_\infty + \beta_{1t}G_\infty\\
    &= G_\infty.
\end{align*}

We next bound $\hat v_t$ similarly. First check $\hat v_0=0$ and $\norm{\hat v_1}=\norm{\beta_2g_1^2}\leq G_\infty^2$. Suppose that $\norm{\hat v_{t-1}} \leq G_\infty^2 $. Then we have
\begin{align*}
    \norm{v_{t}} &= \norm{ (1-\beta_{2}) v_{t-1} + \beta_{2} g_t^2 }\\
    &\leq (1-\beta_{2})\norm{v_{t-1}} + \beta_{2} \norm{g_t^2 }\\
    &\leq (1-\beta_{2})\norm{\hat v_{t-1}} + \beta_{2} \norm{g_t^2 }\\
    &\leq (1-\beta_{2})G_\infty^2 + \beta_{2}G_\infty^2\\
    &= G_\infty^2.
\end{align*}
Thus we complete our proof by observing that $\norm{\hat v_t}\leq \max\{ \norm{\hat v_{t-1}},\norm{v_t}\}\leq G_\infty^2$.

\section{Proof of Theorem \ref{thm:adamStability}}

Different from the regret bound for AMSGrad obtained in~\cite{reddi2019convergence}, our analysis is on the convergence rate. In fact, a slight modification of our proof also provides the convergence rate for AMSGrad for conventional strongly convex optimization, which can be of independent interest. Moreover, our proof avoids the theoretical error in the proof in~\cite{reddi2019convergence} pointed out by~\cite{tran2019convergence}. 
Before proving the theorems, we first provide some useful lemmas.




\begin{lemma}\label{lem:mVBound}\cite[Lemma 2]{reddi2019convergence}
Let $\{m_t, \hat{V}_t\}$ for $t=1,2,\dots$ be sequences generated by Algorithm \ref{alg:opql-AMSGrad}. Given $\alpha_t, \beta_{1t}, \beta_2$ as specified in Theorem \ref{thm:adamStability}, we have
\begin{align*}
    \sum_{t=1}^{T}\alpha_t\norm{ \hat{V}_{t}^{-\frac{1}{4}}m_t }^2 &\leq \frac{\alpha }{(1-\beta_1)(1-\delta)\sqrt{1-\beta_2}}\sum_{i=1}^d\norm{g_{1:T,i}}\sqrt{\sum_{t=1}^T \frac{1}{t}}\\
    &\leq \frac{\alpha \sqrt{1+\log T}}{(1-\beta_1)(1-\delta)\sqrt{1-\beta_2}}\sum_{i=1}^d\norm{g_{1:T,i}}.
\end{align*}
\end{lemma}

\begin{lemma}\label{lem:seqSum}
Let $\alpha_t=\frac{\alpha}{\sqrt{t}}$ and $\beta_{1t} = \beta_1 \lambda^t$ for $t=1,2,\dots$. Then
\begin{equation}
    \sum_{t=1}^T \frac{\beta_{1t}}{\alpha_t} \leq \frac{\beta_1}{\alpha(1-\lambda)^2}.
\end{equation}
\end{lemma}
\begin{proof}
The proof is based on taking the standard sum of geometric sequences.
\begin{equation}
    \sum_{t=1}^T \frac{\beta_{1t}}{\alpha_t} = \sum_{t=1}^T \frac{\beta_{1t}\sqrt{t}}{\alpha} \leq \sum_{t=1}^T \frac{\beta_{1}\lambda^{t-1}t} {\alpha} = \frac{\beta_1}{\alpha} \left( \frac{1}{(1-\lambda)}\sum_{t=1}^T \lambda^{t-1} -T\lambda^T \right) \leq \frac{\beta_1}{\alpha(1-\lambda)^2}.
\end{equation}
\end{proof}

With the lemmas above, we are ready to prove Theorem \ref{thm:adamStability}. Observe that
$$\theta_{t+1} = \Pi_{\mathcal{D}, \hat V_t^{1/4}} \left( \theta_t - \alpha_t \hat{V}_t^{-\frac{1}{2}}m_t\right)=\underset{\theta\in\mathcal{D}}{\min}\norm{\hat V_t^{1/4}\left(\theta_t - \alpha_t \hat{V}_t^{-\frac{1}{2}}m_t - \theta\right)}.$$
Clearly $\Pi_{\mathcal{D}, \hat V_t^{1/4}}(\theta^\star)=\theta^\star$ due to Assumption \ref{asp:boundedDomain}. We start from the update of $\theta_t$ when $t\geq 2$.
\begin{align*}
    \norm{\hat V_t^{1/4}(\theta_{t+1}-\theta^\star)}^2
    &= \norm{\Pi_{\mathcal{D}, \hat V_t^{1/4}}\hat V_t^{1/4}\left(\theta_{t} - \theta^\star - \alpha_t\hat{V}_{t}^{-\frac{1}{2}}m_t\right)}^2\\
    &\leq \norm{\hat V_t^{1/4}\left(\theta_{t} - \theta^\star - \alpha_t\hat{V}_{t}^{-\frac{1}{2}}m_t\right)}^2\\
    &= \norm{\hat V_t^{1/4}(\theta_{t}-\theta^\star)}^2 + \norm{\alpha_t\hat{V}_{t}^{-1/4}m_t}^2 - 2\alpha_t(\theta_t - \theta^\star)^T m_t\\
    &= \norm{\hat V_t^{1/4}(\theta_{t}-\theta^\star)}^2 + \norm{\alpha_t\hat{V}_{t}^{-1/4}m_t}^2 - 2\alpha_t(\theta_t - \theta^\star)^T (\beta_{1t}m_{t-1} + (1-\beta_{1t})g_{t})\\
    &\overset{\text{(i)}}{\leq}  \norm{\hat V_t^{1/4}(\theta_{t}-\theta^\star)}^2 + \norm{\alpha_t\hat{V}_{t}^{-1/4}m_t}^2+ \alpha_t\beta_{1t}\left( \frac{1}{\alpha_t}\norm{\hat V_t^{1/4}(\theta_{t}-\theta^\star)}^2+\alpha_t\norm{ \hat{V}_{t}^{-1/4}m_{t-1} }^2 \right) \\
    &\quad - 2\alpha_t(1-\beta_{1t})(\theta_t - \theta^\star)^T g_{t}\\
    &\overset{\text{(ii)}}{\leq} \norm{\hat V_t^{1/4}(\theta_{t}-\theta^\star)}^2 + \norm{\alpha_t\hat{V}_{t}^{-1/4}m_t}^2 + \beta_{1t}\norm{\hat V_t^{1/4}(\theta_{t}-\theta^\star)}^2+\alpha_t^2\beta_{1t}\norm{ \hat{V}_{t-1}^{-1/4}m_{t-1} }^2\\
    &\quad  - 2\alpha_t(1-\beta_{1t})(\theta_t - \theta^\star)^T g_{t},
\end{align*}
where (i) follows from Cauchy-Schwarz inequality, and (ii) holds because $\hat{v}_{t+1,i} \geq \hat{v}_{t,i},\forall t, \forall i$. 
Next, we take the expectation over all samples used up to time step $t$ on both sides, which still preserves the inequality.
Since we consider i.i.d. sampling case, by letting $\mathcal{F}_t$ be the filtration of all the sampling up to time $t$, we have
\begin{equation}\label{eq:pf2}
    \mathbb{E}\left[(\theta_t - \theta^\star)^T g_{t}\right] = \mathbb{E}\left[\mathbb{E}\left[(\theta_t - \theta^\star)^T g_{t}\right]|\mathcal{F}_{t-1}\right] = \mathbb{E}\left[(\theta_t - \theta^\star)^T \bar g_{t}\right].
\end{equation}

Thus we have
\begin{align*}
    \mathbb{E}\norm{\hat V_t^{1/4}(\theta_{t+1}-\theta^\star)}^2
    &\leq \mathbb{E}\norm{\hat V_t^{1/4}(\theta_{t}-\theta^\star)}^2 + \alpha_t^2\mathbb{E}\norm{\hat{V}_{t}^{-1/4}m_t}^2 + \beta_{1t}\mathbb{E}\norm{\hat V_t^{1/4}(\theta_{t}-\theta^\star)}^2+ \alpha_t^2\beta_{1t}\mathbb{E}\norm{ \hat{V}_{t-1}^{-1/4}m_{t-1} }^2\\
    &\quad  - 2\alpha_t(1-\beta_{1t})\mathbb{E}\left[(\theta_t - \theta^\star)^T g_{t}\right]\\
    &\overset{\text{(i)}}{=} \mathbb{E}\norm{\hat V_t^{1/4}(\theta_{t}-\theta^\star)}^2 + \alpha_t^2\mathbb{E}\norm{\hat{V}_{t}^{-1/4}m_t}^2 + \beta_{1t}\mathbb{E}\norm{\hat V_t^{1/4}(\theta_{t}-\theta^\star)}^2+ \alpha_t^2\beta_{1t}\mathbb{E}\norm{ \hat{V}_{t-1}^{-1/4}m_{t-1} }^2\\
    &\quad  - 2\alpha_t(1-\beta_{1t})\mathbb{E}\left[(\theta_t - \theta^\star)^T \bar g_{t}\right]\\
    &\overset{\text{(ii)}}{\leq} \mathbb{E}\norm{\hat V_t^{1/4}(\theta_{t}-\theta^\star)}^2 + \alpha_t^2\mathbb{E}\norm{\hat{V}_{t}^{-1/4}m_t}^2 + \beta_{1t}\mathbb{E}\norm{\hat V_t^{1/4}(\theta_{t}-\theta^\star)}^2+ \alpha_t^2\beta_{1t}\mathbb{E}\norm{ \hat{V}_{t-1}^{-1/4}m_{t-1} }^2\\
    &\quad  - 2\alpha_t c(1-\beta_{1t})\mathbb{E}\norm{\theta_t - \theta^\star}^2\\
    &\overset{\text{(iii)}}{\leq} \mathbb{E}\norm{\hat V_t^{1/4}(\theta_{t}-\theta^\star)}^2 + \alpha_t^2\mathbb{E}\norm{\hat{V}_{t}^{-1/4}m_t}^2 + \beta_{1t}\mathbb{E}\norm{\hat V_t^{1/4}(\theta_{t}-\theta^\star)}^2+ \alpha_t^2\beta_{1}\mathbb{E}\norm{ \hat{V}_{t-1}^{-1/4}m_{t-1} }^2\\
    &\quad  - 2\alpha_t c(1-\beta_{1})\mathbb{E}\norm{\theta_t - \theta^\star}^2\\
    &\overset{\text{(iv)}}{\leq} \mathbb{E}\norm{\hat V_t^{1/4}(\theta_{t}-\theta^\star)}^2 + \alpha_t^2\mathbb{E}\norm{\hat{V}_{t}^{-1/4}m_t}^2 + G_\infty D_\infty^2\beta_{1t}+ \alpha_t^2\beta_{1}\mathbb{E}\norm{ \hat{V}_{t-1}^{-1/4}m_{t-1} }^2\\
    &\quad  - 2\alpha_t c(1-\beta_{1})\mathbb{E}\norm{\theta_t - \theta^\star}^2,
\end{align*}
where (i) follows from~\Cref{eq:pf2}, (ii) follows due to Assumption \ref{asp:qlearning} and $1-\beta_{1t}>0$, (iii) follows from $\beta_{1t}<\beta_1<1$ and $\mathbb{E}\norm{\theta_t - \theta^\star}^2>0$, and (iv) follows from $\norm{\hat V_t^{1/4}(\theta_{t}-\theta^\star)}^2\leq \norm{\hat V_t^{1/4}}_2^2\norm{\theta_{t}-\theta^\star}^2 \leq G_\infty D_\infty^2$ by Lemma \ref{lem:mvbound} and Assumption \ref{asp:boundedDomain}. We note that (iii) is the key step to avoid the error in the proof in~\cite{reddi2019convergence}, where we can directly bound $1-\beta_{1t}$, which is impossible in~\cite{reddi2019convergence}.
By rearranging the terms in the above inequality and taking the summation over time steps, we have
\begin{align*}
&2c(1-\beta_1)\sum_{t=2}^T\mathbb{E}\norm{\theta_t - \theta^\star}^2\\
&\qquad \leq \sum_{t=2}^T \frac{1}{\alpha_t} \left( \mathbb{E}\norm{\hat V_t^{1/4}(\theta_{t}-\theta^\star)}^2\! -\! \mathbb{E}\norm{\hat V_t^{1/4}(\theta_{t+1}-\theta^\star)}^2 \right) + \sum_{t=2}^T \frac{\beta_{1t}G_\infty D_\infty^2}{\alpha_t}\\
&\qquad \quad +  \sum_{t=2}^{T} \alpha_t \mathbb{E} \norm{ \hat{V}_{t}^{-1/4}m_t }^2 + \sum_{t=2}^{T} \alpha_t\beta_{1} \mathbb{E} \norm{ \hat{V}_{t-1}^{-1/4}m_{t-1} }^2 \\
&\qquad \overset{\text{(i)}}{\leq} \sum_{t=2}^T \frac{1}{\alpha_t} \left( \mathbb{E}\norm{\hat V_t^{1/4}(\theta_{t}-\theta^\star)}^2\! -\! \mathbb{E}\norm{\hat V_t^{1/4}(\theta_{t+1}-\theta^\star)}^2 \right) + \sum_{t=2}^T \frac{\beta_{1t}G_\infty D_\infty^2}{\alpha_t}\\
&\qquad \quad +  \sum_{t=2}^{T} \alpha_t \mathbb{E} \norm{ \hat{V}_{t}^{-1/4}m_t }^2 + \sum_{t=2}^{T} \alpha_{t-1}\beta_{1} \mathbb{E} \norm{ \hat{V}_{t-1}^{-1/4}m_{t-1} }^2\\
&\qquad \leq \sum_{t=2}^T \frac{1}{\alpha_t} \left( \mathbb{E}\norm{\hat V_t^{1/4}(\theta_{t}-\theta^\star)}^2\! -\! \mathbb{E}\norm{\hat V_t^{1/4}(\theta_{t+1}-\theta^\star)}^2 \right) + \sum_{t=2}^T \frac{\beta_{1t}G_\infty D_\infty^2}{\alpha_t}\\
&\qquad \quad +  (1+\beta_1)\sum_{t=1}^{T} \alpha_t \mathbb{E} \norm{ \hat{V}_{t}^{-1/4}m_t }^2,
\end{align*}
where (i) follows from $\alpha_t < \alpha_{t-1}$. With further adjustment of the first term in the right hand side of the last inequality, we can then bound the sum as
\begin{align*}
2c(1&-\beta_1)\sum_{t=2}^T\mathbb{E}\norm{\theta_t - \theta^\star}^2\\
&\leq \sum_{t=2}^T \frac{1}{\alpha_t} \mathbb{E}\left( \norm{\hat V_t^{1/4}(\theta_{t}-\theta^\star)}^2 - \norm{\hat V_t^{1/4}(\theta_{t+1}-\theta^\star)}^2 \right) + \sum_{t=2}^T \frac{\beta_{1t}G_\infty D_\infty^2}{\alpha_t}\\
&\quad + (1+\beta_1)\sum_{t=1}^{T} \alpha_t \mathbb{E} \norm{ \hat{V}_{t}^{-1/4}m_t }^2\\
& = \frac{\mathbb{E}\norm{\hat V_2^{1/4}(\theta_{2}-\theta^\star)}^2}{\alpha_2} + \sum_{t=3}^T \mathbb{E}\left( \frac{\norm{\hat V_t^{1/4}(\theta_{t}-\theta^\star)}^2}{\alpha_t} - \frac{\norm{\hat V_{t-1}^{1/4}(\theta_{t}-\theta^\star)}^2}{\alpha_{t-1}}  \right)\\
&\quad - \frac{\mathbb{E}\norm{\hat V_{T}^{1/4}(\theta_{T+1}-\theta^\star)}^2}{\alpha_{T}} + \sum_{t=2}^T \frac{\beta_{1t}G_\infty D_\infty^2}{\alpha_t} + (1+\beta_1)\sum_{t=1}^{T} \alpha_t\mathbb{E}  \norm{ \hat{V}_{t}^{-1/4}m_t }^2\\
& = \frac{\mathbb{E}\norm{\hat V_2^{1/4}(\theta_{2}-\theta^\star)}^2}{\alpha_2} + \sum_{t=3}^T \mathbb{E}\left( \frac{\sum_{i=1}^d\hat{v}_{t,i}^{1/2}(\theta_{t,i} - \theta_{i}^\star)^2}{\alpha_t} - \frac{\sum_{i=1}^d\hat{v}_{t-1,i}^{1/2}(\theta_{t,i} - \theta_{i}^\star)^2}{\alpha_{t-1}}  \right)\\
&\quad - \frac{\mathbb{E}\norm{\hat V_{T}^{1/4}(\theta_{T+1}-\theta^\star)}^2}{\alpha_{T}} + \sum_{t=2}^T \frac{\beta_{1t}G_\infty D_\infty^2}{\alpha_t} + (1+\beta_1)\sum_{t=1}^{T} \alpha_t \mathbb{E} \norm{ \hat{V}_{t}^{-1/4}m_t }^2.\\
& = \frac{\mathbb{E}\norm{\hat V_2^{1/4}(\theta_{2}-\theta^\star)}^2}{\alpha_2} + \sum_{t=3}^T\sum_{i=1}^d \mathbb{E}(\theta_{t,i} - \theta_{i}^\star)^2\left( \frac{\hat{v}_{t,i}^{1/2}}{\alpha_t} - \frac{\hat{v}_{t-1,i}^{1/2}}{\alpha_{t-1}}  \right)\\
&\quad - \frac{\mathbb{E}\norm{\hat V_{T}^{1/4}(\theta_{T+1}-\theta^\star)}^2}{\alpha_{T}} + \sum_{t=2}^T \frac{\beta_{1t}G_\infty D_\infty^2}{\alpha_t} + (1+\beta_1)\sum_{t=1}^{T} \alpha_t \mathbb{E} \norm{ \hat{V}_{t}^{-1/4}m_t }^2.
\end{align*}
So far we just rearrange the terms in the series sum. Next, we are ready to obtain the upper bound.
\begin{align}
2c(1&-\beta_1)\sum_{t=2}^T\mathbb{E}\norm{\theta_t - \theta^\star}^2 \nonumber\\
&\overset{\text{(i)}}{\leq}\frac{\mathbb{E}\norm{\hat V_2^{1/4}(\theta_{2}-\theta^\star)}^2}{\alpha_2} + D_\infty^2\sum_{t=3}^T\sum_{i=1}^d \mathbb{E}\left( \frac{\hat{v}_{t,i}^{1/2}}{\alpha_t} - \frac{\hat{v}_{t-1,i}^{1/2}}{\alpha_{t-1}}  \right)\nonumber\\
&\quad - \frac{\mathbb{E}\norm{\hat V_{T}^{1/4}(\theta_{T+1}-\theta^\star)}^2}{\alpha_{T}} + \sum_{t=2}^T \frac{\beta_{1t}G_\infty D_\infty^2}{\alpha_t} + (1+\beta_1)\sum_{t=1}^{T} \alpha_t \mathbb{E} \norm{ \hat{V}_{t}^{-1/4}m_t }^2\nonumber\\
&\leq \frac{\mathbb{E}\norm{\hat V_2^{1/4}(\theta_{2}-\theta^\star)}^2}{\alpha_2} + D_\infty^2 \sum_{i=1}^d \mathbb{E} \frac{\hat{v}_{T,i}^{1/2}}{\alpha_T} + \sum_{t=2}^T \frac{\beta_{1t}G_\infty D_\infty^2}{\alpha_t} + (1+\beta_1)\sum_{t=1}^{T} \alpha_t \mathbb{E} \norm{ \hat{V}_{t}^{-1/4}m_t }^2\nonumber\\
& \overset{\text{(ii)}}{\leq} \frac{G_\infty D_\infty^2}{\alpha_2} + \frac{d G_\infty D_\infty^2\sqrt{T}}{\alpha}  + \frac{\beta_1G_\infty D_\infty^2}{\alpha(1-\lambda)^2} + \frac{\alpha (1+\beta_1)\sqrt{1+\log T}}{(1-\beta_1)(1-\delta)\sqrt{1-\beta_2}}\sum_{i=1}^d\mathbb{E} \norm{g_{1:T,i}},   \label{eq:pfRes1}
\end{align}
where (i) follows from Assumption \ref{asp:boundedDomain} and because $\frac{\hat{v}_{t,i}^{1/2}}{\alpha_t} > \frac{\hat{v}_{t-1,i}^{1/2}}{\alpha_{t-1}}$, and (ii) follows from Lemmas \ref{lem:mvbound} - \ref{lem:seqSum}.

Finally, applying the Jensen's inequality yields
\begin{equation}\label{eq:pfthm1}
    \begin{aligned}
    \mathbb{E}\norm{\theta_{out} - \theta^\star}^2 \leq \frac{1}{T}\sum_{t=1}^T\mathbb{E}\norm{\theta_t - \theta^\star}^2.
    \end{aligned}
\end{equation}

We conclude our proof by further applying the bound in~\Cref{eq:pfRes1} to~\Cref{eq:pfthm1}.

\section{Proof of Theorem \ref{thm:adamRestartStability}}

To prove the convergence for AltQ-AMSGradR, the major technical development beyond the proof of Theorem \ref{thm:adamStability} lies in dealing with the parameter restart. More specifically, the moment approximation terms are reset every $r$ steps, i.e., $m_{kr} = \hat{v}_{kr} = 0$ for $k=1,2,\dots$, which implies $\theta_{kr+1} = \theta_{kr}$ for $k=1,2,\dots$.
For technical convenience, we define $\theta_0 = \theta_1$. 
Using the arguments similar to~\Cref{eq:pfRes1}, in a time window that does not contain a restart (i.e. $kr\leq S\leq (k+1)r-1$) we have
\begin{align*}
    2c(1&-\beta_1) \sum_{t=kr}^{S} \mathbb{E}\norm{\theta_t - \theta^\star}^2\\
    &\overset{\text{(i)}}{\leq}  \frac{G_\infty D_\infty^2}{\alpha_{kr+2}} + \frac{dG_\infty D_\infty^2\sqrt{S}}{\alpha} +  \frac{\alpha (1+\beta_1)}{(1-\beta_1)(1-\delta)\sqrt{1-\beta_2}}\sum_{i=1}^d\mathbb{E}\norm{g_{kr+1:S,i}} \sqrt{\sum_{t=kr+1}^{S} \frac{1}{t}}\\
    &\quad + G_\infty D_\infty^2 \sum_{t=kr+2}^{S} \frac{\beta_{1t}}{\alpha_t} +2c(1-\beta_1)\left(  \mathbb{E}\norm{\theta_{kr+1} - \theta^\star}^2 + \mathbb{E}\norm{\theta_{kr} - \theta^\star}^2 \right)\\
    &\overset{\text{(ii)}}{=} \frac{G_\infty D_\infty^2\sqrt{kr+2}}{\alpha} + \frac{dG_\infty D_\infty^2\sqrt{S}}{\alpha} +  \frac{\alpha (1+\beta_1)}{(1-\beta_1)(1-\delta)\sqrt{1-\beta_2}}\sum_{i=1}^d\mathbb{E}\norm{g_{kr+1:S,i}} \sqrt{\sum_{t=kr+1}^{S} \frac{1}{t}}\\
    &\quad + G_\infty D_\infty^2 \sum_{t=kr+2}^{S} \frac{\beta_{1t}}{\alpha_t} +4c(1-\beta_1) \mathbb{E}\norm{\theta_{kr} - \theta^\star}^2,
    \end{align*}
where (i) follows from~\Cref{eq:pfRes1} and (ii) follows from $\theta_{kr+1} = \theta_{kr}$ due to the definition of restart. Then we take the summation over the total time steps and obtain
\begin{align*}
    2c(1&-\beta_1)\sum_{t=1}^T\mathbb{E}\norm{\theta_t - \theta^\star}^2 \\
    &= 2c(1-\beta_1)\left( \sum_{k=1}^{\lfloor T/r\rfloor}\sum_{t=(k-1)r}^{kr-1} \mathbb{E}\norm{\theta_t - \theta^\star}^2 + \sum_{t=\lfloor T/r \rfloor r}^{T} \mathbb{E}\norm{\theta_t - \theta^\star}^2 - \mathbb{E}\norm{\theta_0 - \theta^\star}^2\right)\\
    &\leq \sum_{k=0}^{\lfloor T/r \rfloor }\left( \frac{G_\infty D_\infty^2}{\alpha} \sqrt{kr+2}  +4c(1-\beta_1) \mathbb{E}\norm{\theta_{kr} - \theta^\star}^2  \right) + \sum_{k=1}^{\lfloor T/r \rfloor }\frac{dG_\infty D_\infty^2}{\alpha}\sqrt{kr-1}\\
    &\quad +  \frac{dG_\infty D_\infty^2\sqrt{T}}{\alpha} + \frac{\alpha (1+\beta_1)}{(1-\beta_1)(1-\delta)\sqrt{1-\beta_2}}\sum_{k=1}^{\lfloor T/r \rfloor }\sum_{i=1}^d\mathbb{E}\norm{g_{(k-1)r+1:kr-1,i}} \sqrt{\sum_{t=(k-1)r+1}^{kr-1} \frac{1}{t}}\\
    &\quad +  \frac{\alpha (1+\beta_1)}{(1-\beta_1)(1-\delta)\sqrt{1-\beta_2}}\sum_{i=1}^d\mathbb{E}\norm{g_{\lfloor T/r \rfloor r+1:T,i}} \sqrt{\sum_{t=\lfloor T/r \rfloor r+1}^{T} \frac{1}{t}}\\
    &\quad  +  G_\infty D_\infty^2 \sum_{k=1}^{\lfloor T/r \rfloor }\sum_{t=(k-1)r+2}^{kr-1} \frac{\beta_{1t}}{\alpha_t}+  G_\infty D_\infty^2 \sum_{t=\lfloor T/r \rfloor r+2}^{T} \frac{\beta_{1t}}{\alpha_t} \\
    &\leq \sum_{k=0}^{\lfloor T/r \rfloor }\left( \frac{G_\infty D_\infty^2}{\alpha} \sqrt{kr+2}  +4c(1-\beta_1) \mathbb{E}\norm{\theta_{kr} - \theta^\star}^2  \right) + \sum_{k=1}^{\lfloor T/r \rfloor }\frac{dG_\infty D_\infty^2}{\alpha}\sqrt{kr-1}\\
    &\quad +  \frac{dG_\infty D_\infty^2\sqrt{T}}{\alpha} + \frac{\alpha (1+\beta_1)}{(1-\beta_1)(1-\delta)\sqrt{1-\beta_2}}\sum_{k=1}^{\lfloor T/r \rfloor }\sum_{i=1}^d\mathbb{E}\norm{g_{(k-1)r+1:kr-1,i}} \sqrt{\sum_{t=(k-1)r+1}^{kr-1} \frac{1}{t}}\\
    &\quad +  \frac{\alpha (1+\beta_1)}{(1-\beta_1)(1-\delta)\sqrt{1-\beta_2}}\sum_{i=1}^d\mathbb{E}\norm{g_{\lfloor T/r \rfloor r+1:T,i}} \sqrt{\sum_{t=\lfloor T/r \rfloor r+1}^{T} \frac{1}{t}} + G_\infty D_\infty^2 \sum_{t=1}^{T} \frac{\beta_{1t}}{\alpha_t}.
\end{align*}
We can bound the term $ G_\infty D_\infty^2 \sum_{t=1}^{T} \frac{\beta_{1t}}{\alpha_t}$ by Lemma \ref{lem:seqSum}. Next, we bound another key term in the above inequality. We first observe that $\forall k\geq 2, \forall i\in [d]$,
\begin{align}
     \norm{g_{(k-1)r+1:kr-1,i}} \sqrt{\sum_{t=(k-1)r+1}^{kr-1} \frac{1}{t}}&\overset{\text{(i)}}{\leq} \norm{g_{(k-1)r+1:kr-1,i}} \sqrt{\sum_{t=(k-1)r+1}^{kr-1} \frac{1}{t}}  +|g_{kr,i}|\sqrt{\frac{1}{kr}}\nonumber\\
     &\overset{\text{(ii)}}{\leq} \norm{g_{(k-1)r+1:kr,i}} \sqrt{\sum_{t=(k-1)r+1}^{kr} \frac{1}{t}}, \label{eq:pf1}
\end{align}
where (i) holds due to $|g_{t,i}|\sqrt{\frac{1}{t}}>0$ and (ii) follows from Cauchy-Schwarz inequality. Then we have
\begin{align*}
 &\sum_{k=1}^{\lfloor T/r \rfloor }\sum_{i=1}^d\norm{g_{(k-1)r+1:kr-1,i}} \sqrt{\sum_{t=(k-1)r+1}^{kr-1} \frac{1}{t}} + \sum_{i=1}^d\norm{g_{\lfloor T/r \rfloor r+1:T,i}} \sqrt{\sum_{t=\lfloor T/r \rfloor r+1}^{T} \frac{1}{t}}\\
 &\qquad\overset{\text{(i)}}{\leq} \sum_{k=1}^{\lfloor T/r \rfloor }\sum_{i=1}^d |g_{kr,i}|\sqrt{\frac{1}{kr}} + \sum_{k=1}^{\lfloor T/r \rfloor }\sum_{i=1}^d\norm{g_{(k-1)r+1:kr-1,i}} \sqrt{\sum_{t=(k-1)r+1}^{kr-1} \frac{1}{t}}\\
 &\qquad\quad + \sum_{i=1}^d\norm{g_{\lfloor T/r \rfloor r+1:T,i}} \sqrt{\sum_{t=\lfloor T/r \rfloor r+1}^{T} \frac{1}{t}}\\
 &\qquad= \sum_{k=1}^{\lfloor T/r \rfloor }\sum_{i=1}^d\left( \norm{g_{(k-1)r+1:kr-1,i}} \sqrt{\sum_{t=(k-1)r+1}^{kr-1} \frac{1}{t}} + |g_{kr,i}|\sqrt{\frac{1}{kr}} \right)\\
 &\qquad\quad + \sum_{i=1}^d\norm{g_{\lfloor T/r \rfloor r+1:T,i}} \sqrt{\sum_{t=\lfloor T/r \rfloor r+1}^{T} \frac{1}{t}}\\
 &\qquad\overset{\text{(ii)}}{\leq} \sum_{k=1}^{\lfloor T/r \rfloor }\sum_{i=1}^d \norm{g_{(k-1)r+1:kr,i}} \sqrt{\sum_{t=(k-1)r+1}^{kr} \frac{1}{t}} + \sum_{i=1}^d\norm{g_{\lfloor T/r \rfloor r+1:T,i}} \sqrt{\sum_{t=\lfloor T/r \rfloor r+1}^{T} \frac{1}{t}}\\
 &\qquad= \sum_{i=1}^d\left( \sum_{k=1}^{\lfloor T/r \rfloor } \norm{g_{(k-1)r+1:kr,i}} \sqrt{\sum_{t=(k-1)r+1}^{kr} \frac{1}{t}} + \norm{g_{\lfloor T/r \rfloor r+1:T,i}} \sqrt{\sum_{t=\lfloor T/r \rfloor r+1}^{T} \frac{1}{t}} \right)\\
 &\qquad\overset{\text{(iii)}}{\leq} \sum_{i=1}^d \norm{g_{1:T,i}}\sqrt{\sum_{t=1}^{T} \frac{1}{t}},
\end{align*}
where (i) follows from $|g_{kr,i}|\sqrt{\frac{1}{kr}},\forall k\geq 1,\forall i\in [d]$, (ii) follows from~\Cref{eq:pf1} and (iii) holds due to Cauchy-Schwarz inequality. Then we have

\begin{align*}
    &2c(1-\beta_1)\sum_{t=1}^T\mathbb{E}\norm{\theta_t - \theta^\star}^2 \\
    &\qquad\leq \sum_{k=0}^{\lfloor T/r \rfloor }\left( \frac{G_\infty D_\infty^2}{\alpha} \sqrt{kr+2}  +4c(1-\beta_1) \mathbb{E}\norm{\theta_{kr} - \theta^\star}^2  \right) + \sum_{k=1}^{\lfloor T/r \rfloor }\frac{dG_\infty D_\infty^2}{\alpha}\sqrt{kr-1}\\
    &\qquad\quad +  \frac{dG_\infty D_\infty^2\sqrt{T}}{\alpha} + \frac{\alpha (1+\beta_1)}{(1-\beta_1)(1-\delta)\sqrt{1-\beta_2}}\sum_{k=1}^{\lfloor T/r \rfloor }\sum_{i=1}^d\mathbb{E}\norm{g_{(k-1)r:kr-1,i}} \sqrt{\sum_{t=(k-1)r}^{kr-1} \frac{1}{t}}\\
    &\qquad\quad +  \frac{\alpha (1+\beta_1)}{(1-\beta_1)(1-\delta)\sqrt{1-\beta_2}}\sum_{i=1}^d\mathbb{E}\norm{g_{\lfloor T/r \rfloor r:T,i}} \sqrt{\sum_{t=\lfloor T/r \rfloor r}^{T} \frac{1}{t}} + G_\infty D_\infty^2 \sum_{t=1}^{T} \frac{\beta_{1t}}{\alpha_t}\\
    &\qquad\leq \sum_{k=0}^{\lfloor T/r \rfloor }\left( \frac{G_\infty D_\infty^2}{\alpha} \sqrt{kr+2}  +4c(1-\beta_1) \mathbb{E}\norm{\theta_{kr} - \theta^\star}^2  \right) + \sum_{k=1}^{\lfloor T/r \rfloor }\frac{dG_\infty D_\infty^2}{\alpha}\sqrt{kr-1}\\
    &\qquad\quad +  \frac{dG_\infty D_\infty^2\sqrt{T}}{\alpha} + \frac{\alpha (1+\beta_1)}{(1-\beta_1)(1-\delta)\sqrt{1-\beta_2}}\sum_{i=1}^d\mathbb{E}\norm{g_{1:T,i}}\sqrt{\sum_{t=1}^{T} \frac{1}{t}} + G_\infty D_\infty^2 \sum_{t=1}^{T} \frac{\beta_{1t}}{\alpha_t}\\
    &\qquad\overset{\text{(i)}}{\leq} \sum_{k=0}^{\lfloor T/r \rfloor }\left( \frac{G_\infty D_\infty^2}{\alpha} \sqrt{kr+2}  +4c(1-\beta_1) \mathbb{E}\norm{\theta_{kr} - \theta^\star}^2  \right) + \sum_{k=1}^{\lfloor T/r \rfloor }\frac{dG_\infty D_\infty^2}{\alpha}\sqrt{kr-1}\\
    &\qquad\quad +  \frac{dG_\infty D_\infty^2\sqrt{T}}{\alpha} + \frac{\alpha (1+\beta_1)\sqrt{d(1+\log T)}}{(1-\beta_1)(1-\delta)\sqrt{1-\beta_2}}\sum_{i=1}^d\mathbb{E}\norm{g_{1:T,i}} + \frac{\beta_1G_\infty D_\infty^2}{\alpha(1-\lambda)^2},
\end{align*}
where (i) follows from Lemma \ref{lem:mVBound} and Lemma \ref{lem:seqSum}.

Finally, applying the Jensen's inequality and the above bound, we obtain
\begin{align*}
 \mathbb{E}\norm{\theta_{out} - \theta^\star}^2
 \leq& \frac{1}{T}\sum_{t=1}^T\mathbb{E}\norm{\theta_t - \theta^\star}^2\\
 \leq& \frac{1}{T}\sum_{k=0}^{\lfloor T/r \rfloor }\left( \frac{G_\infty D_\infty^2}{2c\alpha(1-\beta_1)} \sqrt{kr+2}  + 2 \mathbb{E}\norm{\theta_{kr} - \theta^\star}^2  \right) + \frac{1}{T}\sum_{k=1}^{\lfloor T/r \rfloor }\frac{dG_\infty D_\infty^2}{2c\alpha(1-\beta_1)}\sqrt{kr-1}\\
 &+  \frac{dG_\infty D_\infty^2\sqrt{T}}{2c\alpha(1-\beta_1)} + \frac{\alpha (1+\beta_1)\sqrt{d(1+\log T)}}{2c(1-\beta_1)^2(1-\delta)\sqrt{1-\beta_2}}\sum_{i=1}^d\mathbb{E}\norm{g_{1:T,i}} + \frac{\beta_1G_\infty D_\infty^2}{2c\alpha(1-\beta_1)(1-\lambda)^2},
\end{align*}
which concludes the proof.

\end{document}